\documentclass[11pt,twoside]{amsart}
\usepackage[latin1]{inputenc}
\usepackage[T1]{fontenc}
\usepackage{mathtools}
\usepackage{graphicx}
\usepackage{tikz-cd}
\usepackage{tikz}
\usetikzlibrary{chains}

\usepackage[latin1]{inputenc}
\usepackage{amsmath}
\usepackage{amsthm}
\usepackage{amssymb}

\usepackage[all]{xy}
\usepackage{hyperref}
\newtheorem{thm}{Theorem}[section]

\newtheorem{prop}[thm]{Proposition}

\newtheorem{lem}[thm]{Lemma}
\newtheorem{cor}[thm]{Corollary}

\numberwithin{equation}{section}

\theoremstyle{definition}

\newtheorem{remark}[thm]{Remark}

\newcommand{\cal}{\mathcal}
\newcommand{\ka}{{\cal A}}
\newcommand{\kb}{{\cal B}}

\newcommand{\ko}{{\cal O}}
\newcommand{\kp}{{\cal P}}

\newcommand{\ZZ}{\mathbb{Z}}
\newcommand{\QQ}{\mathbb{Q}}

\newcommand{\CC}{\mathbb{C}}

\newcommand{\PP}{\mathbb{P}}

\renewcommand{\to}{\xymatrix@1@=15pt{\ar[r]&}}
\renewcommand{\rightarrow}{\xymatrix@1@=15pt{\ar[r]&}}
\renewcommand{\leftarrow}{\xymatrix@1@=15pt{&\ar[l]}}
\renewcommand{\mapsto}{\xymatrix@1@=15pt{\ar@{|->}[r]&}}
\renewcommand{\twoheadrightarrow}{\xymatrix@1@=18pt{\ar@{->>}[r]&}}
\renewcommand{\hookrightarrow}{\xymatrix@1@=15pt{\ar@{^(->}[r]&}}
\newcommand{\hook}{\xymatrix@1@=15pt{\ar@{^(->}[r]&}}
\newcommand{\congpf}{\xymatrix@1@=15pt{\ar[r]^-\sim&}}
\renewcommand{\cong}{\simeq}

\usepackage{comment}
\excludecomment{comment}
\begin{document}

\title[]{Lagrangian fibrations of hyperk\"ahler fourfolds}

\author[D.\ Huybrechts and C.\ Xu]{Daniel Huybrechts and Chenyang Xu}

\address{DH: Mathematisches Institut and Hausdorff Center for Mathematics,
Universit{\"a}t Bonn, Endenicher Allee 60, 53115 Bonn, Germany}
\email{huybrech@math.uni-bonn.de}

\address{CX: Mathematics Department, MIT,
77 Massachusetts Avenue, 
Cambridge, MA 02139, 
USA\\ \&
BICMR, Beijing, China}
\email{cyxu@mit.edu,  cyxu@math.pku.edu.cn }

\begin{abstract} \noindent
The base surface $B$ of a Lagrangian fibration $X\twoheadrightarrow B$ of a
projective, irreducible symplectic fourfold $X$ is shown to be isomorphic to $\PP^2$.

 \vspace{-2mm}
\end{abstract}

\maketitle
{\let\thefootnote\relax\footnotetext{DH is supported by the SFB/TR 45 `Periods,
Moduli Spaces and Arithmetic of Algebraic Varieties' of the DFG
(German Research Foundation).

CX is partially supported by a Chern Professorship of the MSRI (NSF No. DMS-1440140) and by the
National Science Fund for Distinguished Young Scholars (NSFC 11425101) `Algebraic Geometry'. }
\marginpar{}
}

For a morphism  $X\twoheadrightarrow B$ from
a projective, irreducible symplectic manifold $X$ onto a (normal) variety $B$, 
Matsushita  \cite{MatsFibre,MatsFibre2} proves that only three situations can occur:
The morphism is generically finite, constant, or it describes  a Lagrangian fibration. 

Moreover, it has been generally conjectured that the normal base $B$ of any connected Lagrangian
fibration $X\twoheadrightarrow B$ of a compact, irreducible symplectic manifold $X$
of dimension $2n$  is isomorphic to $\PP^n$, cf.\  \cite[Sec.\ 21.4]{Huy}. The conjecture has been verified
for deformations of Hilbert schemes of K3 surfaces by Markman
\cite{Mark} and for the case that $X$ is projective and $B$ is smooth by Hwang \cite{Hwang}. In dimension
four and assuming smoothness of the base,  the conjecture follows easily
from the ampleness of $\omega_B^*$,
the observation that the Hodge index theorem on $X$ implies $b_2(B)=\rho(B)=1$,
and the classification theory of surfaces,  see  \cite{Marku,MatsFibre}.

By building upon work of Ou \cite{Ou}, the present paper completes
the verification of the conjecture in dimension four:

\begin{thm}\label{thm:main1}
Assume  $X\twoheadrightarrow B$ is a connected Lagrangian fibration
of a projective, irreducible symplectic fourfold $X$ over a normal surface $B$.
Then $B\cong \PP^2$.
\end{thm}

To prove this result, we study the local situation and exclude the case of $E_8$-singularities:

\begin{thm}\label{thm:main2}
Assume $X\to B$ is a projective Lagrangian fibration of a quasi-projective symplectic fourfold over a normal algebraic surface $B$.
If $B$ is locally analytically at a point $0\in B$ of the form $\CC^2/G$ for a finite subgroup $G\subset{\rm GL}(2,\CC)$, then
$G$ is not the binary icosahedral group.
\end{thm}

In general, the normal surface $B$ is known to be $\QQ$-factorial with at most log-terminal
singularities \cite{MatsFibre}. Thus, locally analytically,
$B$ is isomorphic to a quotient of the form $\CC^2/G$ for some finite subgroup $G\subset{\rm GL}(2,\CC)$ (not containing quasi-reflections), cf.\ \cite[Ch.\ 4.6]{Matsuki}. 
As the quotient by the binary icosahedral group is the only factorial quotient singularity \cite{Mumford}, Theorem \ref{thm:main2}
can be expressed equivalently by saying that all singular points of $B$ are non-factorial. Note, once the morphism
is known to be flat, the base is automatically smooth.
In fact, by miracle flatness, smoothness of $B$ is equivalent to flatness of the morphism.

The second theorem implies the first one. Indeed, according to \cite{Ou}, for $X$ projective either $B\cong\PP^2$ or $B$
is a specific Fano surface with exactly one singular point, which, moreover, is an $E_8$-singularity. Note that in the local situation
non-factorial singularities do occur, which makes the proof of the conjecture in dimension four an interesting mix of global and local arguments.

The proof of Theorem \ref{thm:main2} makes use of results of Halle--Nicaise \cite{HaNi}, building upon work of Berkovich
\cite{Berko}, and results of Nicaise--Xu \cite{NiXu}.
The former classifies the essential skeleton of semiabelian degenerations of abelian surfaces, which can be seen as a variant of the results
of Kulikov and Persson, cf.\ \cite{FM}. The latter allows one to identify the essential skeleton and the dual complex of the degeneration. Another crucial input for our arguments
is  Alexeev's work \cite{Alexeev}.

\smallskip

\noindent
{\bf Acknowledgement:}  We are grateful to K.\ Hulek, J.\ Nicaise, W.\ Ou, and I.\ Smith for answering
our questions. Special thanks V.\ Alexeev for explanations concerning \cite{Alexeev} and to Ch.\ Lehn, M.\ Mauri, and W.\ Ou for helpful
comments on the first version of this paper. The first author wishes to thank D.\ Greb, V.\ Lazic, and L.\ Tasin for discussions of the problem at an early stage.
The second author thanks the Institut Henri Poincar\'e and the Institut de Math\'ematiques de Jussieu for their hospitality during his stay in 2018. The first version of the paper was written while the authors enjoyed the hospitality of the MSRI, which is gratefully acknowledged.

While working on this project, F.\ Bogomolov and N.\ Kurnosov published \cite{Bogo} which uses results by J.-M.\ Hwang and K.\ Oguiso  \cite{HwOg} to exclude
the case of $E_8$-singularities.  We are grateful to F.\ Bogomolov for explanations concerning their arguments.
\section{Actions of the binary icosahedral group}
Let $G$ be the binary icosahedral group, which by definition fits into a non-split, short exact sequence
$1\to \{\pm1\}\to G\to A_5\to 1$ with $A_5$ the group of alternating permutations of five letters.
We exploit the properties of $G$ to study its actions in two settings: First, we let $G$ act as a group of homeomorphisms 
on a low-dimensional topological manifold. Later this arises as a $G$-action on the  quotient of the essential skeleton of an abelian surface. Second, we study $G$-actions on complex varieties that are dominated by semiabelian surfaces.

\subsection{}\label{sec:topact}
We begin with the topological setting. Consider a finite extension $1\to\Gamma_0\to\Gamma\to G\to 1$.

\begin{prop}\label{prop:action}
Assume $\Gamma$ acts by homeomorphisms on a topological manifold $S$, which is either 
$S^1\times S^1$, $S^1$, or a point.
Then the induced action of $G$ on $S/\Gamma_0$ is trivial.
\end{prop}

\begin{proof} There is nothing to prove in the last case.
To deal with the case $S\cong S^1\times S^1$, we first observe that if the action of $G$ is not trivial, then the image of 
 $\rho_G\colon G\to {\rm Homeo}(S/\Gamma_0)$ is either $A_5$ or $G$ and, in any case, admits a surjection onto $A_5$.
To see  this, recall  that $\{\pm1\}$ is the only non-trivial normal subgroup of $G$ and that $A_5$ is simple.
 
 We will also use the fact that for any finite subgroup $H\subset {\rm Homeo}(S)$ of orientation-preserving homeomorphisms of a compact 
 real surface $S$,
there exists a complex structure on $S$ with respect to which $H$ is a group of biholomorphic automorphisms \cite[pp.\ 340--341]{Edmonds}.
 In our situation, the complex structure defines a complex curve
  $E\cong\CC/\Lambda$ of genus one. Now, the group of biholomorphic automorphisms of $E$ is a
  semi-direct product of the abelian group $E$ acting by translations
and the group of automorphisms of $E$ as an elliptic curve.  The latter is a subgroup of ${\rm SL}(2,\ZZ)$, which only contains finite subgroups of order at most six.  Hence, $H$ contains an abelian subgroup of index at most six. 

Let us apply this to the image of the given action $\rho_\Gamma\colon \Gamma\to{\rm Homeo}(S)$. Thus, if ${\rm Im}(\rho_\Gamma)$
 contains only orientation-preserving homeo\-morphisms, it contains an abelian subgroup of index at most six. The same then holds for  its image
under the surjection ${\rm Im}(\rho_\Gamma)\twoheadrightarrow{\rm Im}(\rho_G)\twoheadrightarrow A_5$. However, the only non-trivial subgroups
of $A_5$  of index at most six are isomorphic to $A_4$ or $D_{10}$, which are both not abelian.
If ${\rm Im}(\rho_\Gamma)\subset{\rm Homeo}(S)$ does not only contain orientation-preserving homeomorphisms, then the above applies to a certain index two subgroup 
 ${\rm Im}(\rho_\Gamma)'\subset{\rm Im}(\rho_\Gamma)$.  However, as $A_5$ does not contain any subgroup of index two,
 the image of ${\rm Im}(\rho_\Gamma)'$ under ${\rm Im}(\rho_\Gamma)\twoheadrightarrow{\rm Im}(\rho_G)\twoheadrightarrow A_5$ is still $A_5$
 and the arguments above yield again a contradiction.
 
 In the second case, the quotient $S/\Gamma_0\cong S^1/\Gamma_0$ is either
$S^1$ or the closed interval $[0,1]$. Recall that any finite subgroup of the group  ${\rm Homeo}(S^1)$
of homeomorphisms of the circle is either cyclic or dyhedral, cf.\ \cite[Sec.\ 4]{Ghys}. In the case that $S/\Gamma_0$ is again a circle, we apply this to the image of
$\rho\colon G\to {\rm Homeo}(S^1)$.  However, as the binary icosahedral group does not admit any
non-trivial homomorphism onto a cyclic group, the image has to be trivial. Hence, any action of $G$ on $S^1$ is actually trivial.
If $G$ acts by homeomorphisms on the interval $[0,1]$, then it leaves invariant the boundary and hence, by gluing the boundary points, acts by homeomorphisms on $S^1$ as well. However, as before, on $S^1$ the action must be trivial and, hence, $G$ acts trivially on the open interval, which
suffices to conclude.
\end{proof}
\subsection{}\label{sec:geomact} Let us now turn to actions of $G$ on low-dimensional varieties.

\begin{prop}\label{prop:actiongeom}
Let $T$ be a variety of dimension at most two that is rationally dominated by a surface  that is either
abelian, rational, or of the form $\PP^1\times E$ with $E$ an elliptic curve. Then, the binary icosahedral group $G$
does not act freely on $T$.
\end{prop}

\begin{proof}
If $T$ is a curve, its normalization is of genus at most one. 
However, $\PP^1$ does not admit any non-trivial, free group action and, according to Proposition \ref{prop:action},
$G$ does not act (freely) on any elliptic curve. Hence, $G$ does not act freely on $T$.

Let now $T$ be a surface. 
Then, by universality of the minimal resolution, 
the action of $G$ on $T$ lifts to an action of $G$ on its minimal resolution $\tilde T\to T$. Moreover, the action on $\tilde T$ is still free.
Now pick a $G$-equivariant minimal model $\tilde T\to T_0$. Note that $T_0$ is indeed minimal unless, possibly, when it has negative Kodaira dimension.
Suppose the induced $G$-action on $T_0$ has a non-trivial stabilizer $G_x$ at some point $x\in T_0$. Then $G_x$
leaves invariant the exceptional curve in $\tilde T$ over $x$ that is blown-down last. However, this exceptional curve is a $\PP^1$, which does not admit any non-trivial, free group action. Therefore, also the $G$-action on $T_0$ must be free.

If $T$ is a rational surface with a non-trivial, free $G$-action, then the quotient of the free $G$-action on the smooth
rational surface $T_0$ is again a smooth rational surface,  contradicting the fact that such a surface is simply connected.  
 Assume now that the equivariant minimal model $T_0$ of $T$ is a blow-up of a ruled surface over a curve $E_0\coloneqq
{\rm Alb}(T_0)$ of genus one.
Then the action of $G$ on $T_0$ covers an action of $G$ on the base $E_0$ of its ruling. However, again using 
Proposition \ref{prop:action}, $G$ does not admit any non-trivial action on the elliptic curve $E_0$. Hence, $G$ acts on the fibres of the ruling
$T_0\to E_0$, but as before  there are no non-trivial free group actions on $\PP^1$.

Assume now that the minimal model $T_0$ of $T$  is isomorphic, as a complex manifold, to an abelian surface.
It suffices to show that the induced action of $G$ on $T_0$ is trivial. The $G$-action on $T_0$ naturally induces
an action on the Albanese variety ${\rm Alb}(T_0)$, which now is an abelian surface non-canonically isomorphic to $T_0$.
 As $\{\pm 1\}$ is the only non-trivial normal subgroup of $G$, the image of
 $\tau\colon G\to {\rm Aut}({\rm Alb}(T_0))$ is either $G$, $A_5$, or trivial.
 Furthermore, the action on the one-dimensional
$H^{2,0}(T_0)$ is trivial, for $G$ does not admit any non-trivial cyclic quotients. Hence, the action on $H^{1,0}(T_0)$ is special.
Then, according to \cite[Lem.\ 3.3]{Fuji}, the elements of ${\rm Im}(\tau)$ have order $1,2,3,4$, or $6$,
but both groups, $A_5$ and $G$, contain  elements of order five. See also \cite[Cor.\ 3.17]{Katsura}. Hence, $\tau$ is trivial and, therefore, 
the action of $G$ on $T_0$ factors through the abelian group of translations. However, $G$ has only trivial abelian quotients.

Let us next consider the case that $T_0$ is a K3  or an Enriques surface. Now use that $\chi(\ko_{T_0})=1$ or $2$, to show
that $T_0$ does not admit any free group action of any group of order $>2$.

Lastly, assume $T_0$ is a bielliptic surface $(E_1\times E_2)/G_0$ and consider 
the action of the natural extension  $1\to G_0\to\tilde G\to G\to 1$ on $E_1\times E_2$. As above, if this action on $H^{1,0}(E_1\times E_2)$ is
special, then $\tilde G$ acts by translations (note that $\tilde G$ again contains an 
element of order five). Hence, the action of $G$ on $T_0$ factors through an abelian group and, therefore, is trivial.
If the action is not special, then replace $\tilde G$ by the kernel of the induced surjection $\tilde G\twoheadrightarrow\ZZ/n\ZZ$,
which still surjects onto $G$, and conclude as before.
\end{proof}

\section{Semiabelian degenerations}

The generic fibre $X_t$ of the Lagrangian fibration $X\twoheadrightarrow B$  is a smooth variety isomorphic to an abelian surface. Indeed,
its cotangent bundle is isomorphic to its normal bundle and, therefore, trivial. Thus, the fibration 
 can be viewed as a degeneration of  abelian surfaces
to more singular fibres. Very little can be said about arbitrary degenerations of abelian surfaces, so we will have to
construct a new family  that has only mildly singular fibres and, in order to exploit the properties of the binary icosahedral group $G$,
that comes with a  $G$-action.
However, only one-dimensional degenerations of abelian surfaces can be studied by means of essential skeleta and dual complexes.
So, in a second step, we will pass from the two-dimensional family of complex abelian surfaces to a one-dimensional family of abelian surface over a
function field.

\subsection{} Assume $B$ is \'etale locally at $0\in B$ of the form $V\coloneqq U/G$ with $G\subset {\rm GL}(2,\CC)$
an arbitrary finite group acting on a smooth surface $U$ with a unique fixed point $0\in U$ and such that the action is free on $U\setminus\{0\}$. To prove the theorem, we
may actually reduce to the case $B=V=U/G$, which we will henceforth assume.

The action of $G$ on $U$ naturally lifts to an action on the fibre product $X_U\coloneqq U\times_VX$ and, by functoriality, to an action on
its normalization  $Y\to X_U$. Note that, by construction, $Y|_{U\setminus\{0\}}\cong X_{U\setminus\{0\}}$ and, due to the following lemma, $Y$ is in fact smooth.

\begin{lem}
The composition $\varphi\colon Y\to X_U\to X$ is \'etale and can be identified with the quotient of the action of $G$ on $Y$, which is free:
$$\varphi\colon Y\twoheadrightarrow Y/G\cong X.$$
\end{lem}

\begin{proof}
Indeed, as the action of $G$ is free on the complement of $0\in U$, the morphism $\varphi\colon Y\to X$ is \'etale on the complement of
the subscheme $Y_0\subset Y$, which as a Lagrangian subvariety has codimension two \cite{MatEqui}.  However, as $Y$ is normal and $X$ is smooth,  the ramification
locus is pure of codimension one and hence empty, i.e.\ $\varphi$ is \'etale. The induced morphism $Y/G\to X$ is
injective and birational and, using normality of both varieties, in fact an isomorphism.
\end{proof}
From the lemma one immediately deduces the following key fact.

\begin{cor}\label{cor:free}
The induced action $G\times Y_0\to Y_0$ on the central fibre $Y_0\subset Y$ is free.\qed
\end{cor}

Eventually, the free action on the fibre $Y_0$ will lead to a contradiction if the group $G$ is the binary icosahedral group.
This will be done in two steps:
\begin{enumerate}
\item[(i)] The $G$-action on $Y_0$ leaves invariant a certain subvariety $T$.
\item[(ii)] The subvariety $T$ does not admit a free $G$-action.
\end{enumerate}

\subsection{}\label{sec:constr1}
Information about  $Y_0$ together with its $G$-action is difficult to obtain. We will instead work with a functorial
extension of the smooth part of the family $Y\to U$ to a family of stable semiabelic pairs \cite{Alexeev}. The extension
has to be produced in a $G$-equivariant fashion.

In a first step, we  turn the smooth fibres $Y_t$ of $Y\to U$ into stable semiabelic pairs and then later  extend this smooth family to a family over the entire
 $U$.
To realize this first step,  we need to choose, in a uniform way, for each smooth fibre $Y_t$ an effective ample divisor $D_t\subset Y_t$  
and an origin $o(t)\in Y_t$ that makes $Y_t$  an abelian surface.
The first is easy to arrange, uniformly over $U$, by picking an ample divisor $D_X\subset X$. If necessary, shrink $V$ to ensure that $D_X$ does not contain any fibres,
so that its restriction to all fibres is a divisor. Its pull-back under $\varphi$ 
yields a $G$-invariant relative ample effective divisor $D\coloneqq\varphi^{-1}(D_X)\subset Y$.  To turn the smooth part of $Y\to U$ into a family of abelian surfaces,
which amounts to choosing uniformly a zero-section, we need to pass to an appropriate
Galois cover $U'\to U$ (shrink $U$ if necessary), e.g.\ the Galois closure of a multisection of $Y\to U$. Then the smooth part of
$Y'\coloneqq U'\times_UY\to U'$ together with the induced section $o\colon U'\to Y'$  constitutes a family of abelian surfaces and the pull-back of $D\subset Y$ is a
relative ample effective divisor $D'\subset Y'$. Furthermore, the $G$-action on $Y\to U$ lifts to an action of $G'$ on $Y'\to U'$,
where $G'$ is an extension $1\to H\to G'\to G\to 1$ of $G$ by the Galois group $H$ of the cover $U'\to U$.
We summarize the construction above by the following diagram
$$\xymatrix{D'\subset \ar@{}[r(0.6)]|>{}\!\!\!\!\!\!\!\!\!\!\!\!\!\!\!\!&\ar@(ur,ul)[]_{G'}Y'\ar[d]<-1pt>\ar[r]&\ar@(ur,ul)[]_{G}Y\ar[d]\ar[r]&X_U\ar[d]\ar[r]&X\ar[d]<-19pt>\cong Y/G\\
&\ar@(dr,dl)[]^{G'}U'\ar@/^1pc/[u]^o\ar[r]_{/H}&U\ar@{=}[r]&U\ar[r]_-{/G}&V\cong U/G.}$$

Note that when passing to $U'$, the origin $0\in U$ gets replaced by an orbit $\{0_1,\ldots,0_n\}$ of the action of $H$
or, equivalently, of the action of $G'$.

Let $\Delta\subset U'$ be the discriminant locus of $Y'\to U'$. Then the restriction of $Y'$, $D'$, and the section $o$
to its complement $U'\setminus \Delta$ describes a family of abelian surfaces together with an effective polarization of a certain
degree $d$. It corresponds to a morphism to the moduli stack $\ka\kp_{2,d}$ of such pairs,
which admits a coarse quasi-projective moduli space $AP_{2,d}$:
\begin{equation}\label{eqn:classmap}U'\setminus\Delta\to\ka\kp_{2,d}\to AP_{2,d}.
\end{equation}

According to \cite{Alexeev}, the main component $\overline{\ka\kp}_{2,d}$  of the moduli stack of stable semiabelic pairs admits a coarse moduli space
$\overline{AP}_{2,d}$, which is a proper algebraic space and, in fact, a projective variety \cite{KoPat}. Thus, $\overline{AP}_{2,d}$ provides
a compactification of $AP_{2,d}$.
As the surface $U'$ is normal, we may shrink $U'$ to an  open, invariant neighbourhood of the $H$-orbit $\{0_i\}$, such that
(\ref{eqn:classmap}) extends to a morphism 
\begin{equation}\label{eqn:classmap2}
U'\setminus\{0_i\}\to \overline{AP}_{2,d}.
\end{equation}
This determines a stable semiabelic pair for each point $t\in U'\setminus\{0_i\}$, but not quite a family yet.
In order to complete our family over $U'\setminus\Delta$ to a family of stable semiabelic varieties over $\Delta$ and, in particular,
over the points $0_i$,
 we need to modify $U'$.
We do this in two steps. Firstly, resolve the indeterminacies of the rational map
(\ref{eqn:classmap2}) by passing to a (multiple) blow-up $\xymatrix{U'&\ar[l]{\rm Bl}(U')\ar[r]&\overline{AP}_{2,d}\,,}$which can be done in a $G'$-equivariant manner
(with the trivial action on $\overline{AP}_{2,d}$). However, a morphism to $\overline{AP}_{2,d}$ does still not provide us with a family. But there exists a  Galois cover $U''\to{\rm Bl}(U')$, with some Galois group $H'$, by an irreducible normal surface $U'$ (to simplify, you may shrink $U'$ further)
together with a lift $\tilde\phi$ of $\phi\colon{\rm Bl}(U')\to\overline{AP}_{2,d}$ to the stack $\overline{\ka\kp}_{2,d}$:
\vskip-0.5cm
$$\xymatrix{\ar@(ur,ul)[]_{G''}U''\ar[r]_-{/H'}\ar[d]_--{\tilde\phi}&\ar@(ur,ul)[]_{G'}{\rm Bl}(U')\ar[dr]^{\phi\phantom{GHGH}}\ar[r]^-\sigma&\ar@(ur,ul)[]_{G'}U'\ar@{-->}[d]&\ar@{_(->}[l]U'\setminus\{0_i\}\ar[d]\\
\overline{\ka\kp}_{2,d}\ar[rr]&&\overline{AP}_{2,d}&\ar@{_(->}[l]AP_{2,d}.}$$
Here, $G''$ is the natural extension  $1\to H'\to G''\to G'\to 1$ of $G'$ by the Galois group $H'$. Alternatively, $G''$ can be seen as a finite
extension 
\begin{equation}\label{eqn:DefK}
1\to K\to G''\to G\to 1
\end{equation} of the original $G$ by a finite group $K$, which in turn  is an extension of $H$ by $H'$.
Again, $\tilde\phi$ is $G''$-equivariant, this time with respect to a non-trivial action on the target.

Pulling-back the universal family over $\overline{\ka\kp}_{2,d}$ via $\tilde\phi$ to $U''$ yields a
family $(\ka,\kb)\to U''$ of stable semiabelic varieties. Here, $\ka\to U''$ is a semiabelian scheme with an
action on the projective family $\kb\to U''$. We are suppressing  the
effective relative ample divisor, which is of no relevance to us. The existence of the semiabelian extension $\ka$ is due to Grothendieck and
Mumford \cite[Exp.\ IX, Prop.\ 3.5]{Groth} and Faltings--Chai \cite[Ch.\ 6]{FC}, cf.\ \cite{CM}.

 Note that over the pre-image under
$U''\to U'$ of the open subset $U'\setminus\Delta$ the two families $\ka$ and $\kb$ coincide with the pull-back 
$Y''\coloneqq U''\times_{U'}Y'$ of $Y'\to U'$. Clearly, the action of $G'$ on $Y'$ lifts to an action of $G''$ on $Y''$, which in turn
yields an action of $G''$ on the restriction of $(\ka,\kb)$ to the pre-image of $U'\setminus \Delta'$.
In fact, as $\overline{\ka\kp}_{2,d}$ is separated, the action of $G''$ extend to all of $(\ka,\kb)\to U''$ over the
given action on $U''$.

\subsection{}\label{sec:invcurve} 
Consider the equivariant birational morphism $\sigma\colon{\rm Bl}(U')\to U'$ with respect to the action 
of the stabilizer $G'_1\coloneqq {\rm Stab}_{G'}(0_1)$ of the point $0_1\in U'$.
Then, running the equivariant MMP for surfaces, one always finds
a closed point $t_1\in \sigma^{-1}(0_1)$ fixed under the $G'_1$-action or an irreducible $G'_1$-invariant curve $C_1\subset\sigma^{-1}(0_1)$. 
This can also be seen as a special case of a more general result due to Hogadi--Xu \cite[Thm.\ 1.3]{HoXu}, cf.\  \cite[Prop.\ 3.6]{LiuXu}.
In the case of a fixed point $t_1$, we blow-up once more to reduce to the case of a
$G'_1$-invariant exceptional curve $C_1$.
 Note, for later reference, that $G_1'$ is a finite extension $1\to K_1\to G_1'\to G\to 1$, i.e.\ the projection induces a surjection
$G_1'\twoheadrightarrow G$. Note that in general the inclusion $K_1\subset H$ is proper.

Next, we pick  an irreducible curve $C\subset U''$ dominating $C_1$ with its scheme-theoretic generic point $\eta_C\in C$
and consider the subgroup $H'_C\subset H'$ leaving it invariant. Then there exists a subgroup $G''_C\subset G''$
given by an extension $1\to H'_C\to G''_C\to G'_1\to 1$, whose induced action  on $(\ka,\kb)\to U''$ leaves $C\subset U''$ invariant.

Consider now a morphism ${\rm Spec}(R)\to U''$ from the spectrum of a DVR $R$ with its closed point $0_R$ mapped to $\eta_C$ and its generic point $\eta_R$ mapped to the generic point $\eta_{U''}$ of $U''$. We assume that
both extensions $k(\eta_C)\subset k(0_R)$ and $k(\eta_{U''})\subset k(\eta_R)$ are finite and Galois.
Geometrically, up to finite cover, we think of ${\rm Spec}(R)$ as a curve in $U''$ intersecting the curve $C$ in its generic point.

The pull-back of $(\ka,\kb)$ to ${\rm Spec}(R)$ shall be denoted
by $(\ka_R,\kb_R)$. Then, $\ka_R\to{\rm Spec}(R)$ describes a degeneration of the
abelian surface $A\coloneqq\ka_{\eta_R}$ to the semiabelian variety $A_0\coloneqq\ka_{0_R}$. 
After passing to a further finite extension of $R$, we may assume that the irreducible components
of $\kb_{0_R}$ are geometrically irreducible. For later use, we state the following observation.

\begin{lem}\label{lem:classfibre}
Let $W_0\subset\kb_{0_R}$ be an irreducible component of the closed fibre of
$\kb_R\to{\rm Spec}(R)$. Then the $k(0_R)$-variety $W_0$ is either an abelian surface, or birational to $\PP^1\times E$ with $E$ an elliptic curve, or a rational surface.
\end{lem}

\begin{proof}
Indeed, the pair $(\ka_{0_R},\kb_{0_R})$ forms a stable semiabelic variety. In particular, the semiabelian surface $\ka_{0_R}$ acts with only finitely many orbits
on $\kb_{0_R}$. The three cases correspond to the three possibilities for the dimension $0\leq d\leq 2$ of the toric part of $\ka_{0_R}$.
\end{proof}

\subsection{}\label{sec:Actionext}

Note that the normal surface
$U''/K$ comes with a birational morphism $U''/K\to U$, where $K$ is given by (\ref{eqn:DefK}).
Also,
the two quotients $\ka/K$ and $\kb/K$ over $U''/K$ are birational to $Y\to U$. In fact, they coincide over the pre-image of the  complement of the discriminant locus
of the latter.

The action of the group $G''_C$ on $(\ka,\kb)$ lifts to an action of a group $\Gamma$
on $(\ka_R,\kb_R)\to{\rm Spec}(R)$. Here, $\Gamma$ is as an extension $1\to {\rm Gal}\to
\Gamma\to G''_C\to 1$ by the finite Galois group of the extension $k(\eta_{U''})\subset k(\eta_R)$. The action of $\Gamma$ respects the closed
fibre $\kb_{0_R}$. However, as it involves the action of  ${\rm Gal}$, it is not an action on the $k(0_R)$-variety $\kb_{0_R}$.
For later use, we remark that $\Gamma$ can also be regarded as a finite extension of $G$, say $$1\to\Gamma_0\to\Gamma\to G\to 1.$$
The situation is pictured by the following diagram
$$\xymatrix{{\rm Spec}(k(0_R))\ar@(ur,ul)[]_{\Gamma}\ar[r]\ar@{^(->}[d]&C\ar@{^(->}[d]\ar@(ur,ul)[]_{G''_C}\ar[r]_{/H_C'}&C_1\ar@{^(->}[d]\ar@(ur,ul)[]_{G'_1}&\\
{\Gamma}\curvearrowright{\rm Spec}(R)~~
Ê\ar[r]&U''\ar[r]_{/H'}&{\rm Bl}(U')\\
{\rm Spec}(k(\eta_R))\ar@{^(->}[u]\ar[r]&{\rm Spec}(k(\eta_{U''})).\ar@{^(->}[u]&}$$

Let $K_C\coloneqq K\cap G''_C$, which sits in an exact sequence $1\to K_C\to G''_C\to G\to 1$,  and let $\Gamma_C\subset\Gamma$ be its pre-image, for which
we have an exact sequence $1\to{\rm Gal}\to\Gamma_C\to K_C\to 1$.
Then, taking quotients of the semiabelic family $\kb$ yields inclusions
$$\kb_{0_R}/{\rm Gal}\subset \kb_C \text{ and }\kb_{0_R}/\Gamma_C\subset\kb_C/K_C\subset\kb/K\xymatrix{\ar@{<-->}[r]&} Y.$$
In particular, the closure of the image of an irreducible  $k(0_R)$-subvariety $W_0\subset\kb_{0_R}$ yields a $\CC$-subvariety $\overline W_0\subset\kb_C$ fibred over $C$.

\begin{remark}\label{rem:invW0}
Assume $W_0\subset\kb_{0_R}$ is invariant under a subgroup $\Gamma'\subset \Gamma$ which under the projection
$\Gamma\twoheadrightarrow G$  surjects onto $G$. Then $\overline W_0\subset \kb_C$  is $G''_C$-invariant  and $$W\coloneqq \overline W_0/K_C\subset \kb_C/K_C\subset\kb/K\xymatrix{\ar@{<-->}[r]&} Y$$ is  $G$-invariant.
 Also note that the varieties $W=\overline W_0/K_C\subset \kb_C/K_C$ are both fibred over the curve $C/K_C=C_1/K_1$, which is an exceptional curve of ${\rm Bl}(U')/H\to U$.
\end{remark}

\begin{remark}\label{rem:TWdom}
Note that if the $k(0_R)$-variety $W_0$ is of dimension two, i.e.\ an irreducible component of the special fibre
$\kb_{0_R}$, then $W=\overline W_0/K_C$ is of complex dimension three and, hence, a divisor in $\kb/K$. Therefore, $W$  rationally dominates
a subvariety $T$ of the closed fibre $Y_0$.

More concretely, as $C/K_C\subset U''/K$ is blown-down under the birational map
$U''/K\cong{\rm Bl}(U')/H \to U$ and $Y$ is a family over $U$, any fibre $W_t$ of $W\to C/K_C$ for $t$ in a dense open subset of $C/K_C$ rationally dominates  $T$. Now,  specializing the $k(0_R)$-variety $W_0$ to the complex variety $W_t$,
Lemma \ref{lem:classfibre} allows one to conclude that $T$ is dominated by a surface that is either abelian, rational, or isomorphic to $\PP^1\times E$ with $E$ an elliptic curve.
\end{remark}

\section{Action on the skeleton and the fibre}
In this final section, we apply the results of Section \ref{sec:topact}  to the $G$-action on the essential
skeleton of the one-dimensional degeneration over ${\rm Spec}(R)$ constructed above.
This will provide us with a $G$-invariant subvariety of the special fibre. Then, combining the results of Sections \ref{sec:geomact} 
and \ref{sec:constr1},  leads to a contradiction.

\subsection{}
The next result is a consequence of \cite[Cor.\ 4.3.3]{HaNi}, building upon \cite[Sect.\ 6.5]{Berko}, adapted to our situation.
As in the previous section, $A$ denotes the generic fibre $\ka_{\eta_R}$ of the semiabelian family
$\ka_R\to{\rm Spec}(R)$ over a DVR. By construction, the function field  $k(\eta_R)$ of $R$ is  a finite Galois extension
of the function field $k(\eta_{U''})$ of the complex surface $U''$. In the course of our discussion, we will tacitly replace $R$ by a suitable
finite extensions when necessary. Passing to the algebraic closure of $k(0_R)$ and the
completion of $R$, the base change of $A$ yields an abelian surface
over a complete discretely valued field with algebraically closed residue field of characteristic zero. In this situation, the Kontsevich--Soibelman
essential skeleton,  a subspace of the Berkovich space associated with the abelian surface, was studied by \ Musta\c t\u a and Nicaise \cite{MuNic}, see also  \cite{NiSur} for a survey. Suppressing the passage to the completion in the notation, we will write ${\rm Sk}(A)$ for the essential skeleton
of the base change of $A$. Note that ${\rm Sk}(A)$ can be computed in terms of any sncd model \cite[Thm.\ 4.5.5]{MuNic} and is, therefore, a finite CW complex. In particular,
it can in fact be computed over a finite extension of $R$.

\begin{prop}
The group $\Gamma$ acts continuously on the essential skeleton ${\rm Sk}(A)$ of the fibre $A= \ka_{\eta_R}$. The homeomorphism
type of ${\rm Sk}(A)$ is given by one of the three possibilities:
\begin{equation}\label{eqn:skelHN}
{\rm (i)~ }{\rm Sk}(A)\cong\{{\rm pt}\},~{\rm  (ii)~ }{\rm Sk}(A)\cong S^1,~{\rm  or~Ê (iii)~ }{\rm Sk}(A)\cong S^1\times S^1.
\end{equation}  
\end{prop}

\begin{proof} The description of the homeomorphism type of ${\rm Sk}(A)$ is \cite[Cor.\ 4.3.3]{HaNi}. For the first assertion, 
we use again that ${\rm Sk}(A)$ can be computed in terms of an sncd model and of the dual complex of its closed fibre.
Such a model can always be equivariantly constructed, starting with $\kb_R$ and possibly enlarging the group $\Gamma$ further,
so that the group action on $A$ can be extended to the closed fibre. Although the group action on the closed fibre
does not respect the structure of the closed fibre as a variety over $k(0_R)$, it still acts continuously on its dual complex.
\end{proof}

Combined with Proposition \ref{prop:action}, one obtains the following result.

\begin{cor}\label{cor:trivactSk}
If $G$ is the binary icosahedral group, then the induced action on ${\rm Sk}(A)/\Gamma_0$ is trivial.\qed
\end{cor}

We will also need the identification of the essential skeleton with the dual complex of an appropriate sncd model. The result is a consequence
of \cite{NiXu} and \cite{KNX}. Recall that the family $\kb_R\to{\rm Spec}(R)$, with generic fibre $A=\ka_{\eta_R}\cong
\kb_{\eta_R}$, was constructed in a way such that 
all irreducible components of the closed fibre $\kb_{0_R}$ are geometrically irreducible. However, $\kb_{0_R}\subset\kb_R$ is usually not
an snc divisor and its dual complex is not well defined. To remedy the situation, we first observe the following.

\begin{lem}
The family $\kb_R\to{\rm Spec}(R)$ satisfies the following conditions: $\kb_{0_R}$ is reduced, the canonical bundle
$K_{\kb_R}$ is trivial, and the pair $(\kb_R,\kb_{0_R})$ is log-canonical.
\end{lem}

\begin{proof}
By construction $\kb_{0_R}$ is a semiabelic variety and, therefore, by definition reduced. The other two assertions follow from the explicit construction of the degeneration, see \cite[Thm.\ 5.7.1]{Alexeev} and \cite{Mumford}. See also \cite[Lem.\ 4.1\&4.2]{AlNa} for the fact that $\kb_{0_R}$ is Gorenstein
and $K_{\kb_{0_R}}$ trivial and \cite[Lem.\ 3.7\&3.8]{Alexeev2} combined with \cite[Thm.\ 4.9(2)]{Kollar}
for the fact that $(\kb_R,\kb_{0_R})$ is log-canonical.
\end{proof}

Now we use the existence of a $\Gamma$-equivariant dlt modification $\pi\colon\tilde\kb\to\kb_R$ of the pair $(\kb_R,\kb_{0_R})$, cf. \cite[Thm.\ 1.34]{Kollar}.
In particular, $(\tilde\kb,(\tilde\kb_{0_R})_{\rm red}=\pi_*^{-1}\kb_{0_R}+E)$ is dlt.
Here,  $\pi^{-1}_*\kb_{0_R}$ is the strict transform of the (reduced) fibre $\kb_{0_R}$ and $E$ is the reduced
exceptional divisor of $\pi$. Hence, the dual complex $\Delta((\tilde\kb_{0_R})_{\rm red})$ 
of $(\tilde\kb_{0_R})_{\rm red}$ is well defined \cite{dFKX}.  As $(\kb_R,\kb_{0_R})$ is log-canonical,
the dlt modification satisfies $\pi^*(K_{\kb_R}+\kb_{0_R})=K_{\tilde \kb}+\pi^{-1}_*\kb_{0_R}+E$.
As $K_{\kb_R}$ and $\kb_{0_R}$ are both trivial divisors on $\kb_R$, also $K_{\tilde \kb}+\pi^{-1}_*\kb_{0_R}+E$ is trivial.
Hence, both assumptions of  \cite[Thm.\ 24]{KNX} are fulfilled, which immediately yields the next result.

\begin{cor}\label{cor:SkDual}
There exists a $\Gamma$-equivariant homeomorphism ${\rm Sk}(A)\cong\Delta((\tilde\kb_{0_R})_{\rm red})$ between the essential skeleton
of $A$ and the dual complex of the closed fibre $(\tilde\kb_{0_R})_{\rm red}$.\qed
\end{cor}

Let now $W_0\subset\kb_{0_R}$ be an irreducible component and let $\Gamma'\subset \Gamma$ be the subgroup that leaves $W_0$ invariant.
Combining Corollary \ref{cor:trivactSk}, Corollay \ref{cor:SkDual}, and Remark \ref{rem:invW0}, we have proved the following consequence.

\begin{cor}
The composition $\Gamma'\subset\Gamma\twoheadrightarrow G$ is surjective and, therefore, $W=\overline W_0/K_C$ is a $G$-invariant subvariety
of $\kb_C/K_C\subset\kb/K$.\qed
\end{cor}

\subsection{} As the $G$-actions on $\kb/K$ and $Y$ are compatible under the birational map
$\xymatrix{\kb/K\ar@{<-->}[r]&Y,}$the subvariety $T\subset Y_0$ corresponding to $W$,
cf.\ Remark \ref{rem:TWdom}, is $G$-invariant. However, due to Proposition \ref{prop:actiongeom}, the action of $G$ on $T$
cannot be free, which contradicts Corollary \ref{cor:free}.

This concludes the proof of Theorem \ref{thm:main2} and, at the same time, of Theorem \ref{thm:main1}.


\end{document}